\documentclass[11pt,a4paper]{article}
\usepackage{amsmath,amssymb,amsthm}
\usepackage{graphicx}

\usepackage{amsmath, amsthm, amsfonts, amssymb, multicol, pstricks}

\newtheorem{theorem}{Theorem}[section]
\newtheorem{corollary}[theorem]{Corollary}
\newtheorem{lemma}[theorem]{Lemma}
\newtheorem{prop}[theorem]{Proposition}
\theoremstyle{definition}
\newtheorem{definition}[theorem]{Definition}
\newtheorem{example}[theorem]{Example}
\newtheorem{remark}{Remark}[section]
\numberwithin{equation}{section}

\newcommand{\real}{\mathbb{R}}
\newcommand{\rn}{\real^{n}}

\let\leq\leqslant

\let\geq\geqslant

\numberwithin{equation}{section}
\allowdisplaybreaks

\begin{document}

\title{On stability for generalized linear differential equations and applications to impulsive systems}
\author{Claudio A. Gallegos\thanks{ Supported by FONDECYT Postdoctorado No 3220147. Universidad de Chile, Departamento de Matem\'aticas. Casilla 653, Santiago, CHILE. E-mail: {\tt  claudio.gallegos.castro@gmail.com}}, Gonzalo Robledo V.\thanks{Partially supported by FONDECYT Regular No 1210733 . Universidad de Chile, Departamento de Matem\'aticas. Casilla 653, Santiago, CHILE. E-mail: {\tt grobledo@uchile.cl}.}
}

\date{}
\maketitle

\begin{abstract}
In this paper, we are interested in investigating notions of stability for generalized linear differential equations (GLDEs). Initially, we propose and revisit several definitions of stabi\-li\-ty and provide a complete characterisation of them in terms of upper bounds and asymptotic behaviour of the transition matrix. In addition, we illustrate our stability results for GLDEs to linear periodic systems and linear impulsive differential equations. Finally, we prove that the well known definitions of uniform asymptotic stability and variational asymptotic stability are equivalent to the global uniform exponential stability introduced in this article.


\end{abstract}

\pagestyle{myheadings} \markboth{\hfil C. A. Gallegos, G. Robledo V. \hfil $\hspace{3cm}$ } {\hfil$\hspace{1.5cm}$
	{Stability for generalized linear differential equations}
	\hfil}

\section{Introduction}
The beginning of the generalized ordinary differential equations (GODEs) date back to 1957 when the seminal work of J. Kurzweil \cite{JK} introduced this new class of differential equations and gives the first steps in a thoroughgoing construction of a qualitative theory developed in a series of subsequent articles, see \cite{JK3,JK1,JK2}. Kurzweil's contribution was and has been until today an inspiration for several mathematicians who continue growing up this qualitative theory and establishing noteworthy connections between other types of differential systems, such as impulsive differential equations, differential equations in measure, dynamic equations on time scales and functional differential equations, see for instance \cite{FS,FMS,IV,OV,SCHWABIK1,S}. In the last years a lot of progress has been made in the characterisation of exponential dichotomies and its applications \cite{BFS,BFS2}, along with the development of a stability theory \cite{AFT,FGMT1,GGM}. This article will be focused on this last topic.

Let us to commence by introducing the notion of generalized equation above mentioned. Given a function $F: \Omega \to \rn$, where $\Omega=\mathcal{O}\times\real$, and $\mathcal{O}\subset\rn$ is an open set; a function $x:[a,b]\to \rn$, with $[a,b]\subset\real$, is called a {\it solution of the generalized ODE}
\begin{equation}\label{G1}
\dfrac{dx}{d\tau}=DF(x,t) 
\end{equation}
on the interval $[a,b]$, if $(x(t),t)\in\Omega$ for every $t\in[a,b]$, and
\begin{equation}\label{G2}
x(d)-x(c)=\int_{c}^{d}DF(x(\tau),t), \quad \textnormal{whenever $[c,d] \subseteq [a,b]$}.
\end{equation}
It is important to emphasize two aspects of this concept: at first, equation \eqref{G1} is defined via its solution. Secondly, the integral on the right-hand side of \eqref{G2} is understood in the \emph{Kurzweil} sense defined in \cite{KS}. A precise statement will be given in Section~2 to understand the relationship between $F$ and the formalism $DF$ on \eqref{G1} and \eqref{G2}.

In this paper, we focus our attention in the stability theory for GODEs \eqref{G1} described by functions $F\colon\Omega\to\rn$ defined by $F(x,t)=A(t)x+g(t)$, where $A:[0,+\infty)\to\mathcal{L}(\rn)$ and $g:[0,+\infty)\to\rn$ are functions of locally bounded variation with additional properties that will be specified later. This particular case of GODE is known as generalized linear differential equation (GLDE), and is symbolically denoted by 
\begin{equation}\label{G3}
\dfrac{dx}{d\tau}=D[A(t)x+g(t)].
\end{equation}

When the above function $g$ is considered identically null, equation \eqref{G3} is said \emph{homogeneous GLDE} and is denoted by 
\begin{equation}\label{G4}
\dfrac{dx}{d\tau}=D[A(t)x].
\end{equation}
This type of linear problem was intensively studied since 1971 by \v{S}. Schwabik, who developed --in collaboration with M. Tvrd\'y-- a qualitative linear theory in a saga of papers devoted to the existence and uniqueness theorems, the representation of the unique solution for the homogeneous GLDE \eqref{G4} with initial condition $x(s_0)=x_0\in\rn$ through a transition matrix $U(t,s)$, a variation of constants formula for GLDEs \eqref{G3}, among other interesting results,  see {\it e.g.} \cite{SchwF,SCHWABIK1,Schw1,ST,SCHWABIK3}. 

In 2019, Federson \emph{et al.} \cite{FGMT1} have been considering stability notions in the sense of Lyapunov for GODEs \eqref{G1} --see Definition~\ref{stability} in Section 3-- and also established results about uniform stability and uniform asymptotic stability by using Lyapunov functions, see \cite[Th.3.4]{FGMT1} and \cite[Th.3.6]{FGMT1}, respectively. Later, Gallegos \emph{et al.} in \cite{GGM}, established new results concerned to stability, asymptotic stability and exponential stability for GODEs \eqref{G1}, see \cite[Th.3.4]{GGM}, \cite[Th.3.6]{GGM} and \cite[Th.3.9]{GGM}, respectively. In the aforementioned articles the authors applied the above stability results to measure differential equations and dynamic equations on time scales. Recently, Andrade da Silva \emph{et al.} in \cite{AFT} proved a converse Lyapunov theorem on uniform stability, and a converse Lyapunov theorem on uniform boundedness for GODEs \eqref{G1}, see  \cite[Th.4.7]{AFT} and \cite[Th.5.2]{AFT}, respectively. In addi\-tion, this last article studies the relationship between uniform stability, uniform boundedness and stability with respect to perturbations.   

Motivated by the aforementioned articles of stability concerned to GODEs \eqref{G1} \cite{AFT,FGMT1,GGM}, the main purpose in this paper consists in to extend our knowledge about the stability notion introduced in \cite{FGMT1} by considering the special case of the homogeneous GLDE \eqref{G4}, with the goal to obtain sharper results in comparison with the nonlinear framework. More specifically, in Section 3 we provide our first results:
\begin{itemize}
\item In Theorem~\ref{US}, we characterise the uniform stability in terms of the boundedness of the transition matrix $U(t,s)$ associated to the homogeneous GLDE \eqref{G4}. 
\item In Theorem~\ref{UAS}, we characterise the uniform asymptotic stability with a uniform exponential decay of the transition matrix of the form 
\[
\|U(t,s)\|\leq Ke^{-\alpha(t-s)}, \qquad \text{for $K,\alpha>0$, and any $t\geq s\geq0$}.
\]
\item We introduce definitions of global asymptotic stability and global uniform exponential stability for the trivial solution of the homogeneous GLDE \eqref{G4} --see Definition~\ref{gstability}-- 
and we provide a characterisation of these both concepts of stability in terms of the transition matrix $U(t,s)$, see Theorem~\ref{GAS} and Theorem~\ref{UAS}. 
\item  By using the Floquet theory \cite{SchwF} for periodic systems \eqref{G4}, we provide a necessary and sufficient condition ensuring global uniform exponential stability for \eqref{G4}, see Corollary~\ref{Floq}. 
\end{itemize}

In Section 4, we recall --see \cite[Example 6.20]{SCHWABIK1}-- that a linear impulsive differential equation can be seen as a particular case of the GLDE \eqref{G4}, whose transition matrices are the same. This fact allow us to compare the stability notions and results developed in Section 3 with those established in the impulsive linear framework \cite{Bainov,Halanay,Samoilenko}. More specifically, we provide several scalar examples which illustrates the types of stability described in the previous section. See Examples \ref{exUS} to \ref{exGAS-sof}.

In the last section, we are concerned with the notion of \emph{variational  stability} introduced by \v{S}. Schwabik in \cite{SCHWABIK2}, where a well known result states that the variational stability is equivalent to the existence of a uniform bound for the transition matrix $U(t,s)$ and --by using the results from Section 3-- we point out that this is also equivalent to the uniform stability. These pre\-li\-mi\-nary results motivate the main result of this section, namely Theorem \ref{VAS}, which states that the variational asymptotic stability of the trivial solution for the homogeneous GLDE \eqref{G4} is equivalent to the uniform asymptotic stability studied in \cite{AFT,FGMT1,GGM} and revisited in Section 3 for the linear case.

\section{Preliminaries}

Throughout the text, $X$ will always denote a Banach space endowed with a norm $\|\cdot\|_{X}$, and the set $\mathcal{L}(\rn)$ denotes the vector space consisting of all $n\times n$ - matrices with real components endowed with the operator norm. In order to give a brief overview of the Kurzweil integral theory, we will introduce the following definitions:

\begin{itemize}
	\item A subset $P=\{\alpha_0,\alpha_1,...,\alpha_{\nu(P)}\}\subset [a,b]$, with $\nu(P)\in\mathbb{N}$, is said to be a \textit{partition} of $[a,b]$ if $\alpha_0=a<\alpha_1<...<\alpha_{\nu(P)}=b$. We denote the set which contains all partitions of $[a,b]$ by $\mathcal{P}([a,b])$.
	\item The pair  $(P,\tau):=\{([\alpha_{j-1},\alpha_j], \tau_j): j=1,...,\nu(P)\}$, where $P\in\mathcal{P}([a,b])$ and $\tau_j\in[\alpha_{j-1},\alpha_j]$, is called a \emph{tagged partition} of $[a,b]$.
	\item {Any positive function} $\delta:[a,b]\to\mathbb{R}^{+}$ is called a {\itshape gauge} on $[a,b]$.
	\item If $\delta$ is a gauge on $[a, b]$, a
	tagged partition $(P,\tau)$ is called  $\delta$-{\itshape fine} if
	\[
	[\alpha_{j-1},\alpha_j] \subset \left(\tau_j-\delta(\tau_j), \tau_j+\delta(\tau_j)\right), \; \; j=1, \ldots ,\nu(P).
	\]
	\item  Let $f:[a,b]\to X$ be a function. We denote the \textit{variation} of $f$ over the interval $[a,b]$ by
	\[
	\text{var}_{a}^{b}(f)=\sup_{P\in \mathcal{P}([a,b])}\sum_{j=1}^{\nu(P)}\|f(\alpha_{j})-f(\alpha_{j-1})\|_{X}.
	\]
\end{itemize}
The vector space consisting of all functions $f\colon [a,b]\to X$ for which $\text{var}_{a}^{b}(f)<\infty$ is denoted by $BV([a,b],X)$, and it is a Banach space for the norm
\[
\|f\|_{BV}=\|f(a)\|_{X}+\text{var}_{a}^{b}(f).
\]

We denote by $BV_{loc}([0,+\infty),X)$ the set consisting of all functions $f\colon[0,+\infty)\to X$ of locally bounded variation, \emph{i.e.} all functions $f$ such that $\text{var}_{a}^{b}(f)<\infty$ for every compact interval $[a,b]\subset [0,+\infty)$.

Now, we are in a position to define the Kurzweil integral introduced in \cite{JK}:

\begin{definition}\label{Kint}
 A function $G \colon [a,b]\times[a,b]\to X$ is called \emph{Kurzweil integrable} on $[a,b]$, if there is an element $ \mathcal{I} \in X$ having the following property:
for every $\varepsilon>0$, there is a gauge $\delta(\cdot)$ on $[a,b]$ such that
\[
\left\|\sum_{j=1}^{\nu(P)}[G(\tau_j,\alpha_j)-G(\tau_j,\alpha_{j-1})]- \mathcal{I} \right\|_{X}<\varepsilon,
\]
for all $\delta$--fine tagged partition $(P,\tau)$ of $[a,b]$. In this case, $\mathcal{I}$ is called the \emph{Kurzweil integral} of $G$ over $[a,b]$ and will be denoted by $\int_{a}^{b}DG(\tau,t)$.

If $\int_{a}^{b}DG(\tau,t)$ exists, then we define $\int_{b}^{a}DG(\tau,t)=-\int_{a}^{b}DG(\tau,t)$ and we set $\int_{c}^{c}DG(\tau,t)=0$ for all $c \in[a,b]$.
\end{definition} 

\begin{remark} \rm
	Definition~\ref{Kint} has sense due to the fact that given a gauge $\delta$ on $[a, b]$, there always exists a $\delta$-fine tagged partition $(P,\tau)$ of the
	interval $[a, b]$. See \cite[Cousin's lemma~1.4]{SCHWABIK1} or \cite[Lemma 3.1]{Briend}.
\end{remark}

The Kurzweil integral satisfies the usual properties of integration such as linearity, integrabi\-lity of subintervals, additivity on adjacent intervals, among others  (see Chapter I in \cite{SCHWABIK1}). Let us emphasize certain facts of the Kurzweil integra\-bi\-lity notion. The {\it Henstock-Kurzweil integral} definition --also called {\it gauge integral}-- of a function $x\colon[a,b]\to\real$ is obtained by considering in Definition~\ref{Kint} the function  $G(\tau,t)=x(\tau)t$ for all $(\tau,t)\in[a,b]\times[a,b]$. It is well known that this integral contains the classical ones of Riemann and Lebesgue,  see {\it e.g.} \cite{Briend,KS}. Moreover, by considering the function  $G(\tau,t)=x(\tau)g(t)$, where $g:[a,b]\to\real$, the Kurzweil integral turns out be equivalent to the Perron--Stieltjes integral of the function $x(\cdot)$ with respect to $g$, see \cite{JK}. 

In order to provide an appropriate notation for the subsequent sections, we present the follo\-wing example 

\begin{example}\label{ex0}
Consider functions $x\colon[a,b]\to\rn$ and $A:[a,b]\to\mathcal{L}(\rn)$. By defining $G(\tau,t)=A(t)x(\tau)$ in Definition~\ref{Kint}, we obtain the Kurzweil--Stieltjes integral of the function $x$ respect to the function $A$. This particular case is also called Perron--Stieltjes integral. Note that the integral can be approximated by a Stieltjes sum in the following sense
\[
\int_{a}^{b}DG(\tau,t)=\int_{a}^{b}D[A(t)x(\tau)]\sim \sum_{j=1}^{\nu(P)}[A(s_j)-A(s_{j-1})]x(\tau_j).
\] 
Therefore, we use the following conventional notation: 
\[
\int_{a}^{b}D[A(t)x(\tau)]=\int_{a}^{b}{\rm d}[A(s)]x(s).
\]
For a list of fundamental properties of the Perron--Stieltjes integral, we refer to \cite{MT2,MT1,SCHWABIK}. 
\end{example}

An ubiquitous Banach space in the functional setting of the Perron--Stieltjes integral is given by the set $G([a,b],X)$, consisting of all regulated functions $f:[a,b]\to X$, endowed with the uniform convergence norm $\displaystyle\|f\|_{\infty}=\sup_{t\in[a,b]}\|f(t)\|_{X}$. We point out that $BV([a,b],X)\subset G([a,b],X)$. For a detailed discussion about this space we refer the reader to \cite{Frankova}.  

We recall an existence result for the Perron--Stieltjes integral which will be crucial in the development of generalized linear differential equations.

\begin{theorem}\label{exist1}{\rm (\cite[Proposition~2.1]{MT2})}
	Assume that $A\in BV([a,b],\mathcal{L}(\rn))$ and $f\in G([a,b],\rn)$. Then, the Perron--Stieltjes integral $\int_{a}^{b}{\rm d}[A(s)]f(s)$ exists and we have
	\[
	\left\|\int_{a}^{b}{\rm d}[A(s)]f(s)\right\|\leq\int_{a}^{b}\|f(s)\|{\rm d}[{\rm var}_{a}^{s}(A)]\leq\|f\|_{\infty}{\rm var}_{a}^{b}(A).
	\] 
\end{theorem}

	Given an arbitrary function $\varphi:[a,b]\to X$ in $G([a,b],X)$, throughout this paper we will use the following notations
	\[
	\Delta^{+}\varphi(t) := \varphi(t^{+}) - \varphi(t),\quad \text{ and } \quad \Delta^{-}\varphi(t) := \varphi(t ) - \varphi(t^{-}),
	\]
	where $\varphi(t^{+})$ --respectively $\varphi(t^{-})$-- denotes the right-hand (left-hand) limit at $t$.

\subsection{Generalized ODEs}
We recall the concept of generalized ordinary differential equations (GODEs), along with some well known results.  In the next statements, we assume that $F \colon \Omega \to \rn$ is a given $\rn$--valued function, where $\Omega=\mathcal{O}\times[t_0,+\infty)$, $\mathcal{O}\subset\rn$ is an open subset of $\rn$, and $t_0\geq0$.
\begin{definition}\label{defGODE}
Let $J\subset[t_0,+\infty)$ be a nondegenerate interval. A function $x \colon J\to \rn$, is called a \emph{solution of the generalized ordinary differential equation (GODE)}
\begin{equation}
\dfrac{dx}{d\tau}=DF(x,t)  \label{GODE}
\end{equation}
on the interval $J$, if $(x(t),t)\in\Omega$ for every $t\in J$, and
\begin{equation}\label{intgode}
x(d)-x(c)=\int_{c}^{d}DF(x(\tau),t),
\end{equation}
whenever $[c,d] \subseteq J$.
\end{definition}

The integral on the right-hand side of \eqref{intgode} is understood in the sense of Definition~\ref{Kint}.

\begin{remark}
	As pointed out in \cite[pp.99--100]{SCHWABIK1}, a generalized ordinary differential equation denoted by \eqref{GODE} is a formal equation being defined via its solution. Actually, the derivative symbol involved in \eqref{GODE} does not imply a priori that the function $x(\cdot)$, which is a solution of \eqref{GODE}, possess a derivative in some sense. 
\end{remark}

\begin{prop}{\rm (\cite[Proposition~3.5]{SCHWABIK1})}
	If $x:J\to\rn$ is a solution of the GODE \eqref{GODE} on the interval $J\subset[t_0,+\infty)$, then for every fixed $\gamma \in J$ we have
	\begin{equation}
		x(s)=x(\gamma)+\int_{\gamma}^{s}DF(x(\tau),t), \qquad \text{ for any } s\in J. \label{inteq1}
	\end{equation}
	Furthermore, if a function $x:J\to\rn$ satisfies the integral equation \eqref{inteq1} for some $\gamma \in J$ and all $s\in J$, then $x(\cdot)$ is a solution of the GODE \eqref{GODE}.
\end{prop}

With the purpose to obtain more specific information about the solutions of GODEs, we precise extra conditions over the function $F:\Omega\to \rn$. For that reason, we define a broad class of functions $F$ for which are included the theory of ordinary differential equations, measure differential equations, and impulsive differential equations, among others. See \cite{JK1} or  Chapter~V in \cite{SCHWABIK1}.

\begin{definition}
If there exists a nondecreasing function $h \colon [t_0,+\infty)\to\mathbb{R}$ such that $F \colon \Omega\to \rn$ satisfies
\begin{itemize}
\item[(F1)] $\|F(x,s_2)-F(x,s_1)\|\leq|h(s_2)-h(s_1)|$,
for all $(x,s_i)\in\Omega$, with $i=1,2$,
\item[(F2)] $\|F(x,s_2)-F(x,s_1)-F(y,s_2)+F(y,s_1)\|\leq|h(s_2)-h(s_1)| \|x-y\|$, for all $(x,s_i), (y,s_i)\in\Omega$, with $i=1,2$,
\end{itemize}
then we say that $F$ belongs to the class $\mathcal{F}(\Omega,h)$ or simply $F\in\mathcal{F}(\Omega,h)$.
\end{definition}

\begin{example}\label{ex1} Consider $x_0\in\rn$, and let $[a,b]\subset[t_0,+\infty)$ be a compact interval. Assume that $A\in BV([a,b], \mathcal{L}(\rn))$ and $g\in BV([a,b],\rn)$. Note that for every $(x,t)\in\rn\times [a,b]$, we can define 
	
\begin{equation}\label{Flinear}
		F(x,t)=A(t)x+g(t).
\end{equation}

Consider $\Omega={B_c(x_0)}\times [a,b]$, with $c\geq1$. The function $F$ defined by \eqref{Flinear} belongs to the class $\mathcal{F}(\Omega,h)$, where the function $h:[a,b]\to\real$ is defined by $h(t)=(c+\|x_0\|){\text{var}}_{a}^{t}(A) + {\text{var}}_{a}^{t}(g),$ for all $t\in[a,b]$.
\end{example}

The next lemma describes properties of the solutions of GODEs when the function $F$ satisfies the condition (F1).

\begin{lemma}\label{lemma2}
Let  $F\colon \Omega \to \rn$ be a function that satisfies condition {\rm (F1)}. If $x\colon J\to \rn$ is a solution of
the GODE \eqref{GODE} on the interval $J$, then the inequality
\[
\|x(s_2)-x(s_1)\|\leq |h(s_2)-h(s_1)|
\]
holds for every $s_1, s_2 \in J$.
\end{lemma}
\begin{remark}
	When the function $F$ satisfies the condition (F1) and $x(\cdot)$ is a solution of the GODE \eqref{GODE} on $J$, since the function $h$ is locally of bounded variation on $J$ (nondecreasing function defined on $J$), it follows that $x(\cdot)$ is also locally of bounded variation on $J$. In addition, the continuity points of the function $h$ are continuity points of the solution $x(\cdot)$ of the GODE \eqref{GODE} on $J$.  
\end{remark}
The next theorem concerns with the existence-uniqueness of solutions for GODEs defined on a maximal interval, see \cite{FGM}.

\begin{theorem}\label{eumax}
	Let $F\in\mathcal{F}(\Omega,h) $, where $h \colon [t_0,+\infty)\to\mathbb{R}$ is a nondecreasing left--continuous function. Assume that $\Omega=\Omega_{F}:=\{(x,t)\in\Omega \colon  x + F(x,t^{+})-F(x,t)\in \mathcal{O} \}$. Then, for every $(x_0,s_0)\in\Omega$ there exists a unique maximal solution $x \colon [s_0, \omega(x_0,s_0))\to \rn$ of the GODE \eqref{GODE}, with $x(s_0)=x_0$ and $\omega(x_0,s_0)\leq +\infty$.
\end{theorem}
The condition $\Omega=\Omega_{F} :=\{(x,t)\in\Omega : x + F(x,t^{+})-F(x,t)\in \mathcal{O} \}$ ensures that there are no points in $\Omega$ for which the solution of the generalized ODE \eqref{GODE} can escape of $\mathcal{O}$, see \cite[p.127]{SCHWABIK1}. 

\subsection{Generalized linear differential equations}
In this subsection, we put our attention in the special case of GODEs which was regarded in Exam\-ple~\ref{ex1}. Throughout this paper, we consider that $J=[0,+\infty)$, and $A\colon J\to\mathcal{L}(\rn)$ is a $n\times n$ - matrix valued function of locally bounded variation on $J$.  

As it was explained in Chapter VI of \cite{SCHWABIK1}, if $g\in BV_{loc}(J,\rn)$, then a particular case of the GODE \eqref{GODE} is obtained if we consider the function $F\colon \rn\times J \to \rn$ defined by \eqref{Flinear}. This case in literature is known as generalized linear differential equation (GLDE), and it is denoted by 
\begin{equation}\label{LGODE}
\dfrac{dx}{d\tau}=D[A(t)x+g(t)].
\end{equation}
It follows from Definition~\ref{defGODE} that a function $x:J\to\rn$ is a solution of \eqref{LGODE} on $J$ if for every $a,b\in J$ we have
\[
x(b)-x(a)=\int_{a}^{b}D[A(t)x(\tau)+g(t)],
\]
where the integral is considered in the sense of Definition~\ref{Kint}. Taking into consideration the comments in Example~\ref{ex0}, the integral on the right-hand side can be decomposed as follows
\begin{eqnarray*}
\int_{a}^{b}D[A(t)x(\tau)+g(t)]&=&\int_{a}^{b}D[A(t)x(\tau)]+g(b)-g(a)\\
&=&\int_{a}^{b}{\rm d}[A(s)]x(s) +g(b)-g(a).
\end{eqnarray*}

Summarising, a function $x\colon J\to\rn$ is a solution of \eqref{LGODE} if for every $a,b\in J$ we have 
\[
x(b)-x(a)=\int_{a}^{b}{\rm d}[A(s)]x(s) +g(b)-g(a).
\]

\begin{remark}
	Similarly as in the classical theory of ordinary differential equations, in the case that the function $g:J\to\rn$ is identically null, Eq. \eqref{LGODE} is called \emph{homogeneous GLDE} and it is denoted by
	\begin{equation}\label{LhGODE}
	\dfrac{dx}{d\tau}=D[A(t)x].
	\end{equation}
	
\end{remark}

To develop a qualitative theory for the GLDE \eqref{LGODE}, it is neccesary to introduce conditions over the function $A$. For a detailed discussion about this aspect, see \cite[Proposition~6.4]{SCHWABIK1}. Consider the following condition

\begin{description}
	\item[(D)] The matrix $[I-\Delta^{-}A(t)]$ is invertible for all $t\in (0,+\infty)$ and the matrix $[I+\Delta^{+}A(t)]$ is invertible for every $t\in J$. 
\end{description}

Now, we will see that for every  $(x_0,s_0)\in\rn\times J$ there is a unique solution $x(\cdot)$ of \eqref{LGODE} with initial condition $x(s_0)=x_0$ defined on the whole interval $[s_0,+\infty)$, \textit{i.e.} in the generalized linear case, under appropriate assumptions, we always can prolongate the solution and consider $\omega(s_0,x_0)=+\infty$. 

\begin{theorem}{\rm (\cite[Theorem~6.5]{SCHWABIK1})}
Let $g\in BV_{loc}(J,\rn)$. Assume that $A\in BV_{loc}(J,\mathcal{L}(\rn))$ and {\rm (D)} holds. Then for every $(x_0,s_0)\in\rn\times J$ there exists a unique solution $x(\cdot)\in BV_{loc}(J,\rn)$ of the GLDE \eqref{LGODE} with initial condition $x(s_0)=x_0$.
\end{theorem}
In the next, we define the concept of fundamental matrix for the homogeneous GLDE \eqref{LhGODE}.
\begin{definition}
A matrix valued function $X:J\to\mathcal{L}(\rn)$ is called a \emph{fundamental matrix} of the homogeneous GLDE \eqref{LhGODE} if $X$ satisfies the identity
\begin{equation}\label{meq}
X(t)-X(\xi)=\int_{\xi}^{t}{\rm d}[A(s)]X(s)
\end{equation}
for all $\xi,t\in J$, and if the matrix $X(t)$ is invertible for at least one value $t\in J$.
\end{definition}

\begin{remark}
	The integral on the right-hand side of the equation \eqref{meq} should be understood in the sense of Definition~\ref{Kint}, where $G(\tau,t)=A(t)X(\tau)$. In addition, it can be proved that $X:J\to\mathcal{L}(\rn)$ satisfies \eqref{meq} if, and only if, any column of $X$ is a solution of \eqref{LhGODE}, \textit{i.e.} if every $k$-th column $X_k$ of $X$ satisfies 
	\[
	X_k(t)-X_k(\xi)=\int_{\xi}^{t}{\rm d}[A(s)]X_k(s),
	\]
	for all $t,\xi\in J$. Furthermore, if  $A\in BV_{loc}(J,\mathcal{L}(\rn))$ and {\rm (D)} holds, then every fundamental matrix of \eqref{LhGODE} is invertible for all $t\in J$, see \cite[Theorem~6.12]{SCHWABIK1}. 
\end{remark}

The next theorems provide an essential tool to study the theory of GLDEs. Let us emphasize that, despite to seem analogue to the classical world of nonautonomous linear systems, the proofs of the following results are plenty of bulky technicalities and subtle steps most of them related with properties of the Perron--Stieltjes integral, and follow a completely different approach to the proofs in the classical case of ordinary differential equations. For details, we refer \cite{SCHWABIK1,SCHWABIK2}.   	

\begin{theorem}\label{pU}
 Assume that $A\in BV_{loc}(J,\mathcal{L}(\rn))$ and {\rm (D)} holds. Then there exists a uniquely determined $n\times n$ - matrix valued function $U:J\times J\to\mathcal{L}(\rn)$ such that 
 \begin{equation}\label{FO}
 U(t,s)=I+\int_{s}^{t}{\rm d}[A(r)]U(r,s),
 \end{equation}
 for $t,s\in J$. In addition, the function $U$ has the following properties:
 \begin{itemize}
 	\item[{\rm (a)}] $U(t,t)=I$, for all $t\in J$.
 	\item[{\rm (b)}] For every compact interval $[a,b]\subset J$, there exists a constant $M\geq0$ such that 
 	\begin{multicols}{2}
	\begin{itemize}
	\item[{\rm(i)}] $\|U(t,s)\|\leq M$ for all $t,s\in[a,b]$,
	\item[{\rm(ii)}] {\rm var}$_{a}^{b}(U(t,\cdot))\leq M$ for $t\in[a,b]$,
	\item[{\rm(iii)}] {\rm var}$_{a}^{b}(U(\cdot,s))\leq M$ for $s\in[a,b]$,
\end{itemize}
 	\end{multicols}
 \item[{\rm (c)}] For every $r,s,t\in J$ the relation $U(t,s)=U(t,r)U(r,s)$ holds.
 \item [{\rm (d)}] $U(t,s)\in\mathcal{L}(\rn)$ is invertible for every $t,s\in J$, and the relation $[U(t,s)]^{-1}=U(s,t)$ holds.
 \item [{\rm (e)}] For $t,s\in J$, the following equality holds
 \begin{multicols}{2}
 	\begin{itemize}
 		\item [{\rm(i)}] $U(t^{+},s)=[I+\Delta^{+}A(t)]U(t,s)$,
 		\item [{\rm(ii)}] $U(t^{-},s)=[I-\Delta^{-}A(t)]U(t,s)$,
 		\item [{\rm(iii)}] $U(t,s^{+})=U(t,s)[I+\Delta^{+}A(t)]^{-1}$,
 		\item [{\rm(iv)}] $U(t,s^{-})=U(t,s)[I-\Delta^{-}A(t)]^{-1}$.
 	\end{itemize}
 \end{multicols}
 
 \end{itemize}
\end{theorem}

The $n\times n$ - matrix valued function $U$ defined by \eqref{FO} is called transition matrix of the homogeneous GLDE \eqref{LhGODE}. 
\begin{theorem}\label{euhGLDE}{\rm (\cite[Theorem~6.14]{SCHWABIK1})}
Assume that $A\in BV_{loc}(J,\mathcal{L}(\rn))$ and {\rm (D)} holds. Then for every $(x_0,s_0)\in \rn\times J$, the unique solution $x(\cdot)\in BV_{loc}(J,\rn)$ of the homogeneous GLDE \eqref{LhGODE} with initial condition $x(s_0)=x_0$ is given by
\[
x(t,s_0,x_0)=U(t,s_0)x_0, \qquad \text{ fot any } t\in J,
\]
where $U$ is defined by \eqref{FO}.
\end{theorem}

In the next, we recall the variation of constants formula for the GLDE \eqref{LGODE}.

\begin{theorem}\label{vcf}{\rm (\cite[Theorem~6.17]{SCHWABIK1})}
	Assume that $A\in BV_{loc}(J,\mathcal{L}(\rn))$ and {\rm (D)} holds. Then for every $(x_0,s_0)\in \rn\times J$ and $g\in BV(J,\rn)$, the unique solution of \eqref{LGODE} with initial condition $x(s_0)=x_0$ can be written in the form 
	\begin{equation}\label{vcf1}
	x(t,s_0,x_0)=U(t,s_0)x_0 + g(t)-g(s_0)-\int_{s_0}^{t}{\rm d}[U(t,s)](g(s)-g(s_0)),
	\end{equation}
	for all $t\in J$, where $U$ is the transition matrix defined by \eqref{FO}.
\end{theorem}

\begin{remark}
\label{MT-MF}
	As in the classical theory of nonautonomous linear systems, if $X:J\to\mathcal{L}(\rn)$ is an arbitrary fundamental matrix of the homogeneous GLDE \eqref{LhGODE}, then $U(t,s)=X(t)X^{-1}(s)$, for every $t,s\in J$. Therefore, it follows that the identity \eqref{vcf1} can be rewritten in the form 
	\[
	x(t,s_0,x_0)=X(t)\left(X^{-1}(s_0)x_0 + X^{-1}(t)(g(t)-g(s_0))-\int_{s_0}^{t}{\rm d}[X^{-1}(s)](g(s)-g(s_0))\right).
	\]
	Now, using integration by parts formula \cite[Theorem~I.4.33]{SCHWABIK3}, we get
\begin{equation}\label{ibpf}
\begin{split}
x(t,s_0,x_0)&= X(t)\Bigg[X^{-1}(s_0)x_0 +\int_{s_0}^{t}X^{-1}(s){\rm d}[g(s)-g(s_0)]-\sum_{s_0< \tau \leq t}\Delta^{-}X^{-1}(\tau)\Delta^{-}g(\tau)\\
& \ \ +\sum_{s_0\leq \tau < t}\Delta^{+}X^{-1}(\tau)\Delta^{+}g(\tau)\Bigg],
\end{split}
\end{equation}
or equivalently,
\begin{equation}\label{ibpf2}
\begin{split}
x(t,s_0,x_0)&=U(t,s_0)x_0 +\int_{s_0}^{t}U(t,s){\rm d}[g(s)-g(s_0)]-\sum_{s_0< \tau \leq t}\Delta^{-}U(t,\tau)\Delta^{-}g(\tau)\\
&\ \ + \sum_{s_0\leq \tau < t}\Delta^{+}U(t,\tau)\Delta^{+}g(\tau).
\end{split}
\end{equation} 
\end{remark}
%%%%%%%%%%%%%%%%%%%%%%%%%%%%%%%%%%%%%%%%%%%%%%%%%%%%%%%%%%%%%%%%%%%%%%%%%%%%%%%%%%%%%%%%%%%%%%%%%%%%%%%%%%%%%%%%%%%%%%%%%%%%%%%%%%%%%%%%%%%%%%%%%%%%%%%%%%%%%%%%%%%%%%%%%%%%%%%%%%%%%%%%%%%%%%%%%%%%%%%%%%%%%%%%%%%%%%%%%%%%%%%%%%%%%%%%%%%%%%%%%%%%%%%%%%%%%%%%%%%%%%%%%%%%%%%%%%%%%%%%%%%%%%%%

\section{Stability for generalized linear differential equations}

In this section, we will establish Lyapunov-type stability results for the homogeneous GLDE  \eqref{LhGODE}. Throughout this section, we will consider that $J=[0,+\infty)$, $A\in BV_{loc} (J,\mathcal{L}(\rn))$ and condition (D) holds, the function $F:\rn\times J\to\rn$ of the GODE \eqref{GODE} will be defined by $F(x,t)=A(t)x$, for all $(x,t)\in\rn\times J$, \emph{i.e.} we are concerned with the homogeneous GLDE \eqref{LhGODE}. It is clear from definition that the trivial solution $x\equiv 0$ is a 
solution of the homogeneous GLDE \eqref{LhGODE}. 

	Since Theorem~\ref{euhGLDE}, we can guarantee that the unique solution $x(\cdot,s_0,x_0)$ of the homogeneous GLDE \eqref{LhGODE} with initial condition $x(s_0)=x_0\in\rn$ is defined on the whole interval $[s_0,+\infty)$, and $x(t,s_0,x_0)=U(t,s_0)x_0$, for all $t\geqslant s_0\geq 0$, where $U(t,s_{0})$ is the transition matrix defined by \eqref{FO}.

Let us recall stability concepts for the trivial solution $x\equiv 0$ of the generalized ODE \eqref{GODE}, which were introduced in \cite{FGMT1}. These notions emulate the classical Lyapunov--type stability definitions for nonautonomous ODEs, see for instance \cite[Def.4.1.12]{Burton}. We reformulate the statement adapted for the homogeneous GLDE \eqref{LhGODE}.

\begin{definition}\label{stability}
The trivial solution $x\equiv 0$ of the homogeneous GLDE \eqref{LhGODE} is said to be
\begin{itemize}
\item {\bf Stable}, if for every $s_0\geq 0$ and every $\varepsilon>0$, there exists a $\delta=\delta(s_0,\varepsilon)>0$ such that if $x_0\in \rn$ and $\|x_0\|<\delta$, then $\|x(t,s_0,x_0)\|=\|x(t)\|<\varepsilon$ for all $t\in[s_0,+\infty)$.
\item {\bf Uniformly stable}, if it is stable with $\delta>0$ independent of $s_0$.
\item {\bf Uniformly asymptotically stable}, if it is uniformly stable and if there exists $\delta_0>0$ such that, for each $\varepsilon> 0$,
there is a $T = T (\varepsilon)\geq0$  such that if $s_0\geq 0$ and $x_0\in \rn$ with
\[
\|x_0\|<\delta_0,
\]
then
\[
\|x(t, s_0 , x_0 )\| = \|x(t)\| < \varepsilon,
\]
for all $t\in [s_0 + T(\varepsilon),+\infty)$.
\end{itemize}

\end{definition}
Furthermore, we consider the following notions of global asymptotic stability for GLDEs.
\begin{definition}\label{gstability}
The trivial solution $x\equiv 0$ the homogeneous GLDE \eqref{LhGODE} is said to be
\begin{itemize}
\item {\bf Globally asymptotically stable}, if it is stable and for any $x(s_0)=x_0\in\rn$, we get  $\|x(t,s_0,x_0)\|\longrightarrow 0$ as $t\rightarrow +\infty$.
\item {\bf Globally uniformly exponentially stable}, if there exist constants $\alpha, K >0$ such that for any $x(s_0)=x_0\in\rn$, we get $\|x(t,s_0,x_0)\|<K \left\|x_0\right\| e^{-\alpha(t-s_0)} $ for all $t\in[s_0,+\infty)$.
\item {\bf Globally non uniformly exponentially stable}, if there exist a constant $\alpha>0$ and a nondecreasing function $K\colon [0,+\infty)\to [1,+\infty)$ such that for any $x(s_0)=x_0\in\rn$, we get $\|x(t,s_0,x_0)\|<K(s_{0}) \left\|x_0\right\| e^{-\alpha(t-s_0)}$ for all $t\in[s_0,+\infty)$.
\end{itemize}

\end{definition}
Let us emphasise that the global exponential stabilities --uniform and non uniform-- are particular cases of the global asymptotic stability notion. 
\begin{remark}
	Let $B_c$ be the open ball in $\rn$ centered at the origin with radius $c>0$, and $\Omega=B_{c}\times [t_0,+\infty)$ where $t_0\geq0$. The notions of asymptotical stability and exponential asymptotical stability of the origin have been introduced in \cite[Definition 3.1]{GGM} for nonlinear generalized ODEs of type 
	\[
	\frac{dx}{d\tau}=DF(x,t),
	\]
	where the function $F:\Omega\to\mathbb{R}^n$ belongs to the class $\mathcal{F}(\Omega,h)$ and  satisfies $F(0,t)-F(0,s)=0$ for all $t,s\geq t_{0}$. In this case, the initial conditions are assumed in $B_{\rho}$ with $0<\rho<c$. Nevertheless, when $x\mapsto F(t,x)$ is a linear map, the initial conditions can be localized in a ball with infinite radius. Taking this fact into account, we highlight that the first two global stabilities stated in Definition 3.2 for GLDEs are directly recovered from \cite[Definition 3.1]{GGM}, considering in this last mentioned notion the limit case of infinite radius. For that reason, the adjective global is additioned  in the linear context. Moreover, we emphasize that the global non uniform exponential stability stated in Definition~3.2 is a novel definition in the GLDEs framework but emulates the classical definition in the linear ODE context, namely, the contractive case of the non uniform exponential dichotomy,  see \cite{Barreira,Zhang}.
\end{remark}

We are in a position to establish the first stability results for the homogeneous GLDE \eqref{LhGODE}.

\begin{theorem}\label{US}
The trivial solution of the homogeneous GLDE \eqref{LhGODE} is uniformly stable if and only if the transition matrix $U(t,s_0)$ is bounded for all $t\geq s_0\geq0$. 
\end{theorem}
\begin{proof}
	Assume that the trivial solution of \eqref{LhGODE} is uniformly stable. Let $\varepsilon=1$. There exists $\delta=\delta(1)>0$ such that 
	\[
	\|x(t,s_0,x_0)\|=\|U(t,s_0)x_0\|<1,
	\]
	for all $t\geq s_0\geq0$, and $\|x_0\|<\delta$. Therefore, for an arbitrary $z\in\rn$ with $\|z\|\leq1$, we have $\|x(t,s_0,(\delta /2) z)\|=\|U(t,s_0)(\delta /2) z\|<1$ for all $t\geq s_0\geq0$. Hence, we deduce that 
	\[
	\|U(t,s_0)\|=\sup_{\|z\|\leq1}\|U(t,s_0)z\|\leq 2/\delta,
	\]
	for all $t\geq s_0\geq0$.
	
	On the other hand, suppose that there exists a positive constant $C>0$ such that $\|U(t,s_0)\|\leq C$ for all $t\geq s_0\geq0$. Given $\varepsilon>0$, we can take  $\delta=\varepsilon/C$, and it follows that 
	\[
	\|x(t,s_0,x_0)\|=\|U(t,s_0)x_0\|\leq\|U(t,s_0)\|\|x_0\|<\varepsilon,
	\]
	for all $t\geq s_0\geq0$, and $\|x_0\|<\delta$. Therefore, the trivial solution of \eqref{LhGODE} is uniformly stable. 
\end{proof}

In the previous Theorem~\ref{US}, the boundedness of the transition matrix is independent of the initial time $s_0\in J$. It is straightforward to see that the stability can be characterised with a bound depending on the initial time $s_0\in J$. Specifically, we can state the following result

\begin{corollary}
The trivial solution of the homogeneous GLDE \eqref{LhGODE} is stable if and only if for every $s_0\in J$, there is $C=C(s_0)>0$ such that $\|U(t,s_0)\|\leq C$ for all $t\geq s_0$.  
\end{corollary}

In the next, we characterise the global asymptotic stability of the trivial solution.  

\begin{theorem}\label{GAS}
	The trivial solution of the homogeneous GLDE \eqref{LhGODE} is globally asymptotically stable if and only if for every $s_0\in J$, it follows that $\|U(t,s_0)\|\longrightarrow0$ as $t\to+\infty$.
\end{theorem}
\begin{proof}
Let $s_0\in J$ and assume that the trivial solution of the generalized linear ODE \eqref{LhGODE} is globally asymptotically stable. Then it follows that
\[
\|U(t,s_0)\|=\sup_{\|z\|\leq1}\|U(t,s_0)z\|=\sup_{\|z\|\leq1}\|x(t,s_0,z)\|\longrightarrow0, \qquad \text{ as } t\to+\infty.
\]
On the other hand, suppose that for every $s_0\in J$, $\|U(t,s_0)\|\longrightarrow0$ as $t\to+\infty$. We will prove the stability of the trivial solution. Let $\varepsilon>0$ be fixed. There exists $N\in\mathbb{R}$ such that $\|U(t,s_0)\|<\varepsilon$ for all $t\geq N$. In addition, since $U(\cdot,s_0):[s_0,N]\to\mathcal{L}(\rn)$ is a function of bounded variation, it follows that the function $\|U(\cdot,s_0)\|:[s_0,N]\to\real$ is of bounded variation on $[s_0,N]$. Hence, there exists a $K(\varepsilon,s_0)>0$ such that $\|U(t,s_0)\|<K(\varepsilon,s_0)$ for all $t\geq s_0$.

Now, if we choose $\delta=\varepsilon/K(\varepsilon,s_0)$, then we have
\[
\|x(t,s_0,x_0)\|=\|U(t,s_0)x_0\|=\|U(t,s_0)\|\|x_0\|<\varepsilon,
\]
for all $\|x_0\|<\delta$, and we conclude the stability. In addition, since $x(t,s_0,x_0)=U(t,s_0)x_0$, and $\|U(t,s_0)\|\longrightarrow0$ as $t\to+\infty$, it follows that $x(t,s_0,x_0)\longrightarrow0$ as $t\to+\infty$ (independent if $x_0$ is close to zero).
\end{proof}

The next result is a characterisation of the uniform asymptotic stability for the homogeneous GLDE \eqref{LhGODE} and proves that the uniform asymptotic stability in the linear case is equivalent to the global uniform exponential stability. 

\begin{theorem}\label{UAS}
	The trivial solution of the homogeneous GLDE \eqref{LhGODE} is uniformly asymptotically stable if and only if there exist positive constants $\alpha,K>0$ such that 
\begin{equation}\label{exp}
\|U(t,s_0)\|\leq Ke^{-\alpha(t-s_0)}
\end{equation}
	for all $t\geq s_0\geq0$.
\end{theorem}

\begin{proof}
	Assume that there exist positive constants $\alpha,K>0$ such that \eqref{exp} holds for all $t\geq s\geq0$. Clearly, the transition matrix is bounded, with $\|U(t,s_0)\|\leq K$, for all $t\geq s_0\geq0$. Hence, by Theorem~\ref{US}, the trivial solution of the linear generalized ODE \eqref{LhGODE} is uniformly stable. 
	
	Let $\varepsilon>0$ be fixed. We can always assume that $K>\varepsilon$. Consider $x_0\in\rn$ and $s_0\geq0$. Note that we have the following equivalence
	\[
	 Ke^{-\alpha(t-s_{0})}<\varepsilon\; \Longleftrightarrow \; t>s_{0}+\dfrac{\ln(K/\varepsilon)}{\alpha}.
	\]  
	Therefore, considering $\delta_{0}=1$ and $T(\varepsilon)=\dfrac{\ln(K/\varepsilon)}{\alpha}$, we obtain that
	\begin{equation*}
	\|x(t,s_{0},x_0)\|=\|U(t,s_0)x_0\|\leq \|x_0\|Ke^{-\alpha(t-s_{0})}<\varepsilon,
	\end{equation*}
	for all $t\geq s_0+T(\varepsilon)$. Hence, we conclude that the trivial solution of the generalized ODE \eqref{LhGODE} is uniformly asymptotically stable. 
	
	On the other hand, suppose that the trivial solution of \eqref{LhGODE} is uniformly asymptotically stable. Let $0<\varepsilon<1$ be fixed. There exists $\delta_0>0$ independent of $\varepsilon$, and there is $T=T(\varepsilon)>0$, such that 
	\[
	\|U(t,s_0)x_0\|<\dfrac{\varepsilon\delta_{0}}{2},
	\]  
for all $t\geq s_{0} + T$ and $\|x_0\|<\delta_0$. Now, for an arbitrary $z\in\rn$ with $\|z\|\leq1$ we obtain $\|U(t,s_0)(\delta_{0}/2)z\|<(\varepsilon\delta_{0}/2)$ for all $t\geq s_{0} + T$, which implies that $\|U(t,s_0)z\|<\varepsilon$ for all $t\geq s_{0} + T$. Hence, 
\begin{equation}\label{eq1}
\|U(t,s_0)\|<\varepsilon,
\end{equation}
for all $t\geq s_{0} + T$. 

From the uniform stability assumption of \eqref{LhGODE} and Theorem~\ref{US}, we deduce the existence of a positive constant $C>0$ such that $\|U(t,s_0)\|\leq C$ for all $t\geq s_0\geq0$.

Consider  $t\geq s_0\geq0$, and let $N\geq1$ be the smallest integer such that $t\leq s_0+NT$. Then, by the transition property of the transition matrix we have the following decomposition
\[
U(t,s_0)=U(t,s_0+(N-1)T)U(s_0+(N-1)T,s_0+(N-2)T)\cdots U(s_0+T,s_0).
\]  

By \eqref{eq1} and the previous equality we obtain
\begin{equation}\label{eq2}
\|U(t,s_0)\|\leq \|U(t,s_0+(N-1)T)\|\varepsilon^{N-1}\leq C\varepsilon^{N-1}, \qquad \text{for all } t\geq s_{0}.
\end{equation}
Consider an arbitrary $\alpha>0$. In the case that $N=1$, i.e. $s_0<t\leq s_0+T$, we have 
\[
\|U(t,s_0)\|\leq C=Ce^{\alpha(t-s_0)}e^{-\alpha(t-s_0)}\leq Ce^{\alpha T}e^{-\alpha(t-s_0)}=Ke^{-\alpha(t-s_0)},
\]
where $K=Ce^{\alpha T}$. In addition, in the case that $N>1$, for $s_0+(N-1)T\leq t\leq s_0+NT$, we deduce
\begin{equation}\label{eq3}
e^{-\alpha(t-s_0)}e^{\alpha T}\geq e^{-\alpha(N-1)T}.
\end{equation}

Therefore, if we take $\alpha=-\dfrac{\ln(\varepsilon)}{T}>0$, then from \eqref{eq2} and \eqref{eq3} it follows that  
\begin{eqnarray}
\|U(t,s_0)\|&\leq Ce^{-\alpha(N-1)T}\leq Ce^{\alpha T}e^{-\alpha(t-s_0)}=Ke^{-\alpha(t-s_0)}.
\end{eqnarray}

Gathering all the previous observations, for $\alpha=-\dfrac{\ln(\varepsilon)}{T}>0$ and $K=Ce^{\alpha T}>0$, we obtain
\[
\|U(t,s_0)\|\leq Ke^{-\alpha(t-s_0)}, \qquad \text{for all } t\geq s_0\geq0,
\]
getting the desired result.
\end{proof}

\begin{remark}
\label{rUAS}
	Theorem~\ref{UAS} ensures that the uniform asymptotic stability is equivalent to the global uniform exponential stability. In fact, if the trivial solution of the homogeneous GLDE \eqref{LhGODE} is uniformly asymptotically stable, then there exist positive constants $K,\alpha>0$ satisfying \eqref{exp}. Then, for an arbitrary $s_0\geq0$ and $x_0\in\rn$ we get
	\[
	\|x(t,s_0,x_0)\|=\|U(t,s_0)x_0\|\leq K\|x_0\|e^{-\alpha(t-s_0)}, \qquad \text{for all $t\geq s_0$},
	\]
	and the global uniform exponential stability follows. The other implication follows easily from the second paragraph of the proof of Theorem~\ref{UAS}. It is important to emphasize that this equivalence is well known in the classic context \cite[Th. 58.7]{Hahn}. Nevertheless, to the best of our knowledge, it has not been verified for the homogeneous GLDE \eqref{LhGODE}.   
\end{remark}
In the next, we recall a result of the Floquet theory which has been generalized to the context of GLDEs, see \cite{SchwF} and \cite{SCHWABIK3}. 
\begin{theorem}\label{Floquet1}
	Assume that $A\in BV_{loc}(J,\mathcal{L}(\rn))$, condition {\rm (D)} holds and $A(t+\omega)-A(t)=C$, for all $t\in J$, where $\omega>0$ and $C\in\mathcal{L}(\rn)$ is a constant matrix. Then for the fundamental matrix $X(t)=U(t,0)$ associated to the homogeneous GLDE \eqref{LhGODE}, there exist a $n\times n$ - matrix valued function $P\colon J \to \mathcal{L}(\rn)$ which is $\omega$-periodic and a constant matrix $Q\in\mathcal{L}(\rn)$ such that
	\begin{equation}
	\label{MFF}
	X(t)=P(t)e^{Qt},\qquad \text{ for all $t\in J$},
	\end{equation}
	where $Q=\frac{1}{\omega}\ln\left(X(\omega)\right)$.  
\end{theorem}

\begin{remark}
	Theorem~\ref{Floquet1} can be stated on the whole real axis if we interchange -- in the enunciate -- the interval $J$ by $\real$ and consider the condition (D) as follows: the matrices $[I-\Delta^{-}A(t)]$ and $[I+\Delta^{+}A(t)]$ are invertible for all $t\in\real$. The proof of the previous assertion can be found in \cite[Th.3.2, p.124]{SCHWABIK3}. 
\end{remark}

Similarly as in the classical Floquet theory, the stability of a $\omega$--periodic homogeneous GLDE \eqref{LhGODE} can be studied in terms of the monodromy matrix $X(\omega)$.

\begin{corollary}\label{Floq}
Assume the hypotheses of Theorem \ref{Floquet1}.  The trivial solution of the homogeneous GLDE \eqref{LhGODE} is globally uniformly exponentially stable if and only if the eigenvalues of the monodromy matrix $X(\omega)$
are inside of the unit circle.
\end{corollary}

\begin{proof}
Suppose that every eigenvalue $\rho$ of $X(\omega)$ verifies $|\rho|<1$. We will prove that for any $t\geq s_{0}\geq 0$, the inequality \eqref{exp} holds. Note that as $X(\omega)$ is invertible, it is easy to deduce from \eqref{MFF} that for every $t\in J$ the matrix $P(t)$ is invertible. Now, from Remark \ref{MT-MF}, we have that
\[
U(t,s_{0})=P(t)e^{Q(t-s_{0})}P^{-1}(s_{0}), \quad \textnormal{for any $t\geq s_{0}\geq0$},
\]
where $P^{-1}\colon J\to \mathcal{L}(\rn)$ is also $\omega$--periodic. In fact, note that
$P(t)P^{-1}(t)=P(t+\omega)P^{-1}(t)=I$ leads to $P^{-1}(t)=P^{-1}(t+\omega)$ and the $\omega$--periodicity follows. 

By \eqref{MFF} we have that $P(t)=X(t)e^{-Qt}$. Now, by using Remark \ref{MT-MF} and statement (b) from Theorem \ref{pU} we deduce the existence of
a constant $M\geq 0$ such that for any $t\in [0,\omega]$:
\begin{displaymath}
\begin{array}{rcl}
\|P(t)\| &\leq &  \|U(t,0)\| \|e^{-Qt}\| \\
&\leq & M\,\max\limits_{t\in [0,\omega]}\|e^{-Qt}\|,
\end{array}
\end{displaymath}
and we conclude the existence of $M_{1}>0$ and $M_{2}>0$ such that
$$
\sup\limits_{t\in [0,\omega]}\|P(t)\|\leq M_{1} \quad \textnormal{and} \quad
\sup\limits_{t\in [0,\omega]}\|P^{-1}(t)\|\leq M_{2}, 
$$
where the second inequality can be obtained in a similar way. Now, the above estimations and the 
$\omega$--periodicity of $P$ and $P^{-1}$ implies that
\begin{equation}
\label{EF1}
\begin{array}{rcl}
\|U(t,s_{0})\| &\leq & \|P(t)\|\,\|P^{-1}(s_{0})\|\,   \|e^{Q(t-s_{0})}\|\\
           &\leq &  K_{0} \, \|e^{Q(t-s_{0})}\|,     
\end{array}
\end{equation}
for all $t\geq s_{0}\geq 0$, where $K_{0}=M_{1}M_{2}$. In addition, note that every eigenvalue $\lambda_{k}$ of $Q$, 
has the form
\[
\lambda_{k}=\frac{1}{\omega}[\ln|\rho_{k}|+i\,(\arg \rho_{k}+2m\pi)] \quad 
\textnormal{with $m\in \mathbb{Z}$},
\]
where $\rho_{k}$ is an eigenvalue of the matrix $X(\omega)$, we refer the reader
to \cite[p.20]{Adrianova} for details. Moreover, as all the eigenvalues of $X(\omega)$
are inside the unit circle, from \cite[Th.1.9.2]{Perko} we infer the existence of positive constants $\alpha>0$ and $K_{1}>0$ such that $\|e^{Q(t-s_{0})}\|\leq K_{1}e^{-\alpha(t-s_{0})}$ for all $t\geq s_{0}\geq 0$. This fact combined with \eqref{EF1} leads to the estimate
\[
\|U(t,s_{0})\|\leq Ke^{-\alpha(t-s_{0})} \quad \textnormal{for any $t\geq s_{0}\geq 0$},
\]
where $K=K_{0}K_{1}$. Finally, gathering Theorem \ref{UAS} and Remark \ref{rUAS} we conclude that the trivial solution of the homogeneous GLDE \eqref{LhGODE} is globally uniformly exponentially stable.

On the other hand, suppose that the homogeneous GLDE \eqref{LhGODE} is globally uniformly exponentially stable. We will prove that any eigenvalue $\rho$ of
$X(\omega)$ belongs to the interior of the unit circle. By contradiction, assume that there exists an eigenvalue $\rho$ satisfying $|\rho|\geq 1$ with
$$
U(\omega,0)x_{0}=\rho\,x_{0},
$$
where $x_{0}$ is an eigenvector associated to $\rho$. Then, it is straightforward to prove that
$$
U(k\omega,0)x_{0}=\rho^{k}x_{0}  \quad \textnormal{for any $k\in \mathbb{N}$},
$$
since $U(t+\omega,s+\omega)=U(t,s)$ for all $t,s\geq 0$. The above identity combined with \eqref{exp} leads to
$$
|\rho|^{k} \|x_{0}\|=\|U(k\omega,0)x_{0}\|\leq Ke^{-\alpha k\omega}\|x_{0}\|,
$$
for some $K>0$ and $\alpha>0$. As $|\rho|\geq 1$, a contradiction is obtained letting $k\to \infty$.
\end{proof}

\section{Applications to scalar impulsive equations}
A noteworthy result of Schwabik \cite[pp.193--196]{SCHWABIK1} establishes that every linear impulsive diffe\-rential equation can be written as a particular homogeneous GLDE \eqref{LhGODE} provided some mild conditions. In this section, we will consider the scalar case in order to study the stability properties of the impulsive differential equation
\begin{equation}
\label{IDE}
\left\{\begin{array}{l}
\dot{x}=a(t)x, \qquad \textnormal{for $t\neq \tau_{k}$}\\\\
\Delta^{+}x(\tau_{k})=b_{k}x(\tau_{k}),
\end{array}\right.    
\end{equation}
where $t\geq 0$, the function $a\colon [0,+\infty)\to\real$ is locally Lebesgue integrable on $[0,+\infty)$, $\{\tau_{k}\}_{k\in\mathbb{N}}$ is a divergent sequence without cluster points, and the real sequence $\{b_{k}\}_{k\in\mathbb{N}}$ verifies $1+b_{k}\neq 0$ for all $k\in\mathbb{N}$.

As above mentioned, the impulsive differential equation \eqref{IDE} can be written as a scalar homogeneous GLDE 
\begin{equation}
\label{sGLDE} 
\frac{dx}{d\tau}=D[A(t)x],
\end{equation}
by considering the function $A\colon[0,+\infty)\to\real$ defined by 
\begin{equation}\label{A}
A(t)=\int_{0}^{t}a(r)\,dr+\sum\limits_{k=1}^{\infty}b_{k}H_{\tau_{k}}(t),\qquad \text{for all } t\geq0,
\end{equation}
where $H_{\tau_{k}}(\cdot)$ is defined by $H_{\tau_{k}}(t)=H(t-\tau_{k})$ for all $t\geq0$, and $H(\cdot)$ denotes the Heaviside function. Note that the function $A(\cdot)$ is continuous from the left, locally of bounded variation on $[0,+\infty)$, and satisfies condition (D) from Section~2. Therefore, the unique forward solution of the homogeneous GLDE (\ref{sGLDE}) passing through $x_{0}$ at time $t=s_{0}$ has the form
$x(t,s_{0},x_0)=U(t,s_{0})x_{0}$, for all $t\geq s_{0}\geq 0$, where $U(t,s_{0})$ is the transition matrix described by Theorem~\ref{pU}. 

Following the techniques employed in \cite[pp.195-196]{SCHWABIK1}, we can state that for $t,s\geq0$ the transition matrix for the homogeneous scalar GLDE \eqref{sGLDE} has the form
\begin{equation}
\label{MT}
U(t,s)=\left\{\begin{array}{ccl}
\Phi(t,s) &,&  \tau_{j}<t,s\leq \tau_{j+1},\\
\Phi(t,s)\prod\limits_{k=j+1}^{i}(1+b_{k}) 
&, &
\tau_{j}<s\leq \tau_{j+1}\leq \tau_{i}<t\leq\tau_{i+1},
\\
\Phi(t,s)\prod\limits_{k=j+1}^{i}\left(\frac{1}{1+b_{k}}\right)
&, &
\tau_{j}<t\leq \tau_{j+1}\leq \tau_{i}<s\leq\tau_{i+1},

\end{array}\right.
\end{equation}
where  $\Phi(t,s)=e^{\int_{s}^{t}a(\tau)\,d\tau}$ denotes the transition matrix associated to the continuous part of the impulsive differential equation \eqref{IDE}.

On the other hand, note that the transition matrix \eqref{MT} coincides with the transition matrix corresponding to the impulsive equation \eqref{IDE} which is well known for the classical theory of impulsive linear systems developed
in the pioneering works of \cite{Bainov,Halanay,Lakshmikantham,Samoilenko}. Hence, both the impulsive differential equation \eqref{IDE} and the scalar homogeneous GLDE \eqref{sGLDE} have the same transition matrix, which in turn implies that its solutions passing through $x_0$ at time $s_0$ are the same and are given by 
\[
x(t,s_0,x_0)=U(t,s_0)x_0, \qquad \text{ for all $t\geq s_0\geq0$.}
\]

The stability notions described in Definitions \ref{stability} and \ref{gstability} for homogeneous GLDEs
have some similarities with the stability theory for linear impulsive 
equations. In fact, the concepts of stability, uniform stability and global asymptotic stability of GLDEs are equivalent to those for impulsive systems stated in \cite[p.35]{Milev} and \cite[pp.56--57]{Samoilenko} respectively. Moreover, the stability theory of linear impulsive systems allows to classify several types of global asymptotic stabilities that can be described as a particular case of dichotomies with identity projector. A definition tailored to the scalar impulsive case \eqref{IDE} --adapted from \cite[Def.2.2]{Zhang}-- is the following:
\begin{definition}
\label{NUC}
The scalar impulsive differential equation \eqref{IDE}:
\begin{itemize}
\item[(a)] Has a {\it uniform contraction} if there exist 
a couple $(K,\alpha)$ of positive constants and an increasing function $h\colon [0,+\infty)\to [1,+\infty)$ with
$h(0)=1$ and $\lim\limits_{t\to +\infty}h(t)=+\infty$ such that
\begin{displaymath}
|U(t,s)|\leq K\left(\frac{h(t)}{h(s)}\right)^{-\alpha} \quad \textnormal{for any $t\geq s\geq 0$}.
\end{displaymath}
\item[(b)] Has a {\it non uniform contraction} if there exist three positive constants $\{K,\alpha,\varepsilon\}$ and two increasing functions $h,\mu \colon [0,+\infty)\to [1,+\infty)$ with
$h(0)=\mu(0)=1$ and $\lim\limits_{t\to +\infty}h(t)=\lim\limits_{t\to +\infty}\mu(t)=+\infty$ such that
\begin{displaymath}
|U(t,s)|\leq K\left(\frac{h(t)}{h(s)}\right)^{-\alpha}\mu(s)^{\varepsilon} \quad \textnormal{for any $t\geq s\geq 0$}.
\end{displaymath}
\end{itemize}
\end{definition}

In the next, we illustrate with simple examples our stability results for the homogeneous GLDE \eqref{sGLDE}  established in the previous section and we compare them with the stability notions in the impulsive framework.

\begin{example}\label{exUS}
If we assume that the function $a(\cdot)$ is Lebesgue integrable on $[0,+\infty)$ and the sequence $\{b_{k}\}_{k\in\mathbb{N}}$ is such that the series $\sum\limits_{k=1}^{\infty}\ln(|1+b_{k}|)$ is absolutely convergent, then the origin of the homogeneous GLDE \eqref{sGLDE} is uniformly stable. In fact, the integrability of the function $a\colon[0,+\infty)\to\real$ and the absolute convergence of the series imply the existence of  constants $L_{1}\in\real$ and $L_{2}>0$ such that
\[
\displaystyle \Phi(t,s)=e^{\int_{s}^{t}a(\tau)\,d\tau}\leq e^{L_{1}} \quad \textnormal{and} \quad
\prod\limits_{k=j+1}^{i}|1+b_{k}|\leq e^{\sum\limits_{k=1}^{\infty}|\ln(|1+b_{k}|)|}\leq e^{L_{2}}.
\]

By using \eqref{MT} we obtain that $|U(t,s)|\leq e^{L_{1}+L_{2}}$ for all $t\geq s\geq 0$ and the uniform stability of the null solution of (\ref{sGLDE}) follows from Theorem \ref{US}.
\end{example}

As we mentioned above, the definition
of the uniform stability is the same for 
GLDEs and linear impulsive equations, which is determined by the uniform boundedness of the transition matrix.
Hence, the origin is also a uniformly stable solution for the impulsive equation (\ref{IDE}).

Now, we provide an example of global (non uniform) asymptotic stability for the homogeneous GLDE \eqref{sGLDE}.
\begin{example}\label{ex2}
	Consider the function $a:[0,+\infty)\to\real$ defined by $a(t)=-\omega  -ct\sin(t)$, for all $t\geq0$,
	with $0<c<\omega$, and assume that the series
	$\sum\limits_{k=1}^{\infty}\ln(|1+b_{k}|)$ is absolutely convergent. Then the origin of \eqref{sGLDE} is globally asymptotically stable. In fact, it follows from \cite[Proposition 2.3]{Barreira} that the 
	assumption $0<c<\omega$ implies the existence of a positive constant $K_{0}>0$ such that
	\[
	\Phi(t,s)=e^{\int_{s}^{t}a(\tau)\,d\tau}\leq K_{0}e^{-(\omega-c)(t-s)}e^{2cs}, \qquad \textnormal{for $t\geq s\geq 0$}.
	\]
	Therefore, using \eqref{MT} and the absolute convergence of the series as in Example~\ref{exUS}, we obtain a positive constant $K>0$ such that
\begin{equation}\label{eqexgas}
|U(t,s)|\leq Ke^{-(\omega-c)(t-s)+2cs}=K(s)e^{-(\omega-c)(t-s)},  \qquad \textnormal{for $t\geq s\geq 0$}.
\end{equation}
In consequence, by Theorem \ref{GAS} we have that the origin of the homogeneous GLDE \eqref{sGLDE} is globally asymptotically stable. Moreover, since $\omega>c$, this global stability corresponds to the specific case of global non uniform exponential stability
stated in Definition \ref{gstability}.

\end{example}
In Example \ref{ex2}, the definition
of global non uniform exponential stability for GLDEs can be seen as a specific case of non uniform contraction
of the impulsive differential equation \eqref{IDE} with functions $h,\mu\colon [0,+\infty)\to [1,+\infty)$ defined by $h(t)=e^{t}$ and $\mu(t)=e^{2t}$, while the corresponding positive constants are given by $\alpha=\omega-c$ and $\varepsilon=c$.

The last two examples will illustrate the uniform asymptotic stability of the trivial solution for the homogeneous GLDE \eqref{sGLDE}.

\begin{example}\label{exUAS}
Let  $a(\cdot)$ be a bounded and piecewise continuous function verifying the average condition
\begin{equation}\label{average}
\limsup\limits_{t-s\to +\infty}\frac{1}{t-s}\int_{s}^{t}a(\tau)\,d\tau \leq -\alpha, \quad \textnormal{for $t>s$},
\end{equation}
where $\alpha>0$. Assume further that the series
$\sum\limits_{k=1}^{\infty}\ln(|1+b_{k}|)$ is absolutely convergent. Then the origin of the homogeneous GLDE \eqref{sGLDE} is uniformly asymptotically stable. In fact, the average condition \eqref{average} means that the upper Bohl exponent
of the equation $\dot{x}=a(t)x$ is smaller than $-\alpha<0$ (see \cite{Baravanov}, \cite[pp.258--259]{Hinrichsen} for details), which in turn implies the existence of $K_{0}>0$ such that
\[
\displaystyle \Phi(t,s)=e^{\int_{s}^{t}a(\tau)\,d\tau}\leq K_{0}e^{-\alpha(t-s)} \quad \textnormal{for all $t\geq s\geq 0$}. 
\]

In addition, as in the previous Example~\ref{exUS}, by the absolute convergence of the series we get an estimation for the product $\prod\limits_{k=j+1}^{i}|1+b_{k}|$. Now,  using \eqref{MT} and the above estimations, we deduce the existence of a positive constant $K>0$ such that
$|U(t,s)|\leq Ke^{-\alpha(t-s)}$ for all $t\geq s\geq 0$. Hence, the uniform asymptotic stability follows from Theorem \ref{UAS}. 

\end{example}

Let us emphasise that Example~\ref{exUAS} is also an example of global uniform exponential stability for the trivial solution of the homogeneous GLDE \eqref{sGLDE}. In fact, for an arbitrary $x_0\in\real$ and $s_0\geq0$ we get 
\[
|x(t,s_0,x_0)|=|U(t,s_0)x_0|\leq |x_0|Ke^{-\alpha(t-s_0)}, \qquad \text{for all } t\geq s_0.
\]

Moreover, the definition
of global uniform exponential stability for GLDEs can be regarded as a particular case of uniform contraction
of the impulsive differential equation \eqref{IDE} with function $h\colon [0,+\infty)\to [1,+\infty)$ defined by $h(t)=e^{t}$.

In the above examples, the asymptotic behaviour of the homogeneous GLDE (\ref{sGLDE})
is dominated by the continuous part while the discrete one is harmless. Our last example inverses this behaviour.

\begin{example}\label{exGAS-sof} Consider a positive valued function $a:[0,+\infty)\to [0,+\infty)$ 
Lebesgue integrable on $[0,+\infty)$. Moreover, assume that
the sequence of impulses $\{\tau_{k}\}_{k\in\mathbb{N}}$ verifies the condition
\begin{itemize}
    \item[a)] There exists $c>0$ such that
    \begin{displaymath}
    \inf\limits_{i+1>j}\left\{\frac{i-j}{\tau_{i+1}-\tau_{j}}\right\}=c,
    \end{displaymath}
\end{itemize}
while the sequence $\{b_{k}\}_{k\in\mathbb{N}}$ verifies the condition
\begin{itemize}
    \item[b)] There exists $\theta,\eta\in (0,1)$ such that $\eta<|1+b_{k}|\leq  \theta$ for all $k\in \mathbb{N}$. 
\end{itemize}

In the case that $t>s$ with $s\in (\tau_{j},\tau_{j+1}]$ and $t\in (\tau_{i},\tau_{i+1}]$
with $i\geq j+1$. By using the asbolute integrability of $a(\cdot)$,
combined with properties a) and b), and defining $\alpha_{0}=-c\ln(\theta)>0$ 
we have that:
\begin{displaymath}
\begin{array}{rcl}
|U(t,s)| &\leq & K\, \theta^{i-j}\\
&\leq& K\, e^{\left[\frac{i-j}{\tau_{i+1}-\tau_{j}}\right](\tau_{i+1}-\tau_{j})\ln(\theta)}\\
&\leq&  K\, e^{-\alpha_{0}(\tau_{i+1}-\tau_{j})},\\
&\leq& K\, e^{-\alpha_{0}(t-s)}.
\end{array}
\end{displaymath}

In the case that $t>s$ with $t,s\in(\tau_{j},\tau_{j+1}]$, we choose a $s_0\in(\tau_{j+1},\tau_{j+2}]$ and decompose  $U(t,s)=U(t,s_0)U(s_0,s)$. Then, by using the previous estimations and condition (b), we obtain
\begin{equation*}
\begin{split}
|U(t,s)|&\leq Ke^{-\alpha_{0}(s_{0}-s)}|[U(s_{0},t)]^{-1}|\\
& \leq Ke^{-\alpha_{0}(t-s)}e^{\left|-\int_{s_0}^{s}a(\xi)d\xi\right|}\dfrac{1}{|1+b_{j+1}|}\\
&\leq K^{2}\eta^{-1}e^{-\alpha_{0}(t-s)}.
\end{split}
\end{equation*}
Summarising, there exist positive constants $N>0$ and $\alpha_{0}>0$ such that $|U(t,s)|< Ne^{-\alpha_{0}(t-s)}$, for all $t> s\geq0$. Therefore, by Theorem~\ref{UAS}, we deduce that the trivial solution of the homogeneous GLDE \eqref{sGLDE} is uniformly asymptotically stable. 

\end{example}

\section{Variational stability for generalized linear differential equations}

In this section, we recall the concepts of variational stability and stability with respect to perturbations for homogeneous GLDEs introduced in \cite{SCHWABIK2}. We will study the relation of these concepts with the notion of stability given in the previous Section 3. 

 Throughout this section, we assume that the function $A\in BV_{loc}([0,+\infty),\mathcal{L}(\rn))$ is conti\-nuous from the left and condition (D) holds. Since $A$ is left-continuous, we just  assume that $[I+\Delta^{+}A(t)]$ is invertible for all $t\geq0$ because the case when the matrix $[I-\Delta^{-}A(t)]$ is invertible evidently holds. 

We begin by introducing the notion of variational stability. 
\begin{definition}
The trivial solution $x\equiv 0$ of the homogeneous GLDE \eqref{LhGODE} is said to be

\begin{itemize}\label{defvs}
	\item {\bf Variationally stable}, if for every $\varepsilon>0$, there exists a $\delta=\delta(\varepsilon)>0$ such that if $y:[s_0,+\infty)\to\rn$, $s_0\geq0$, is a function of locally bounded variation on $[s_0,+\infty)$ and left-continuous on $(s_0,+\infty)$ with $\|y(s_0)\|<\delta$ and
	\[
	\sup_{t\geq s_0}\text{var}_{s_0}^{t}\left(y(s)-\int_{s_0}^{s}{\rm d}[A(\xi)]y(\xi)\right)<\delta,
	\]
   then $\|y(t)\|<\varepsilon$ for all $t\geq s_0$.
	\item {\bf Variationally attracting}, if there exists $\delta_0>0$ and for every $\varepsilon>0$ there is a $T=T(\varepsilon)\geq0$ and $\gamma=\gamma(\varepsilon)>0$ such that if $y:[s_0,+\infty)\to\rn$, $s_0\geq0$, is a function of locally bounded variation on $[s_0,+\infty)$ and left-continuous on $(s_0,+\infty)$ with $\|y(s_0)\|<\delta_0$ and
	\[
	\sup_{t\geq s_0}\text{var}_{s_0}^{t}\left(y(s)-\int_{s_0}^{s}{\rm d}[A(\xi)]y(\xi)\right)<\gamma,
	\]
	then $\|y(t)\|<\varepsilon$ for all $t\geq s_0 + T(\varepsilon)$ and $s_0\geq 0$.

	\item {\bf Variationally asymptotically stable}, if it is variationally stable and variationally attrac\-ting.
\end{itemize}
\end{definition}

We cite the following reflection about the variational stability notion stated in \cite[Remark~10.7]{SCHWABIK1}: ``This concept comes from the following intuitive idea: if a certain function $y(\cdot)$ given on some interval $[t_0,t_1]\subset [0,+\infty)$ is such that the initial value $y(t_0)$ is close to zero and on the interval $[t_0,t_1]$ the function $y$ is almost a solution of \eqref{GODE}, i.e. the variation of the function 
\[
y(s)-y(t_0)-\int_{t_0}^{s}DF(y(\tau),t)
\]
on $[t_0,t_1]$ is small enough, then $y$ is close to zero on the interval $[t_0,t_1]$".

\begin{remark}\label{r4.1}
	Note that in Definition~\ref{defvs} when it is considered the function $y(\cdot)$ as a solution of homogeneous GLDE \eqref{LhGODE}, then we recover the Definition~\ref{stability} of stability given in Section 3. In particular, the variational stability notion can be seen as the uniform stability, and the variational asymptotic stability notion can be seen as the uniform asymptotic stability. 
\end{remark}

In the next, we will define a concept of stability respect to perturbations which is closely related with the previous notion of variational stability. For our purpose, consider the following generalized perturbed equation
\begin{equation}\label{linearP}
	\dfrac{dx}{d\tau}=D[A(t)x+P(t)],
\end{equation}
where $P:[0,+\infty)\to\rn$ is a function. 

\begin{definition}
	The trivial solution $x\equiv 0$ of the homogeneous GLDE \eqref{LhGODE} is said to be
	
	\begin{itemize}
		\item {\bf Stable with respect to perturbations}, if for every $\varepsilon>0$, there exists a $\delta=\delta(\varepsilon)>0$ such that if $y_0\in\rn$ with $\|y_0\|<\delta$, and $P\in BV_{loc}([s_0,+\infty))$ is a left-continuous function on $(s_0,+\infty)$ such that
		\[
		\sup_{t\geq s_0}\text{var}_{s_0}^{t}(P)<\delta,
		\]
		then $\|y(t,s_0,y_0)\|<\varepsilon$ for all $t\geq s_0$, where $y(\cdot,s_0,y_0)$ is the unique solution of \eqref{linearP} with initial condition $y(s_0,s_0,y_0)=y_0$.
		\item {\bf Attractive with respect to perturbations}, if there exists $\delta_0>0$ and for every $\varepsilon>0$ there is a $T=T(\varepsilon)\geq0$ and $\gamma=\gamma(\varepsilon)>0$ such that if $y_0\in\rn$ with $\|y_0\|<\delta_0$, and $P\in BV_{loc}([s_0,+\infty))$ is a left-continuous function on $(s_0,+\infty)$ such that
		\[
		\sup_{t\geq s_0}\text{var}_{s_0}^{t}(P)<\gamma,
		\]
		then $\|y(t,s_0,y_0)\|<\varepsilon$ for all $t\geq s_0+T(\varepsilon)$, where $y(\cdot,s_0,y_0)$ is the unique solution of \eqref{linearP} with initial condition $y(s_0,s_0,y_0)=y_0$.
		\item {\bf Asymptotically stable  with respect to perturbations}, if it is stable with respect to perturbations and attrac\-tive with respect to perturbations.
	\end{itemize}
\end{definition}
We continue with the reflection of Schwabik stated in \cite[Remark~10.7]{SCHWABIK1}: ``The stability with respect to perturbations is motivated by the desire that the solutions of the perturbed equation \eqref{linearP} be close to zero on a certain $[t_0,t_1]$ whenever the value $y(t_0)$ is close to zero and the perturbing term $P$ in \eqref{linearP} is small in the sense that var$_{t_0}^{t_1}(P)$ is small".

\begin{remark}\label{r4.2}
	When the perturbation $P$ is identically null, we recover again the notion of stabi\-lity given in Section 3. Specifically, considering $P\equiv0$, the stability with respect to perturbations emulates the uniform stability, and the asymptotic stability with respect to perturbations emulates the uniform asymptotic stability from Definition~\ref{stability}.   
\end{remark}

The following theorem shows the equivalence between the notion of variational stability and stability with respect to perturbations. For a proof of the next result, we refer to the reader \cite[Theorem~10.8]{SCHWABIK}.
\begin{theorem}\label{evsp}
	The trivial solution of the homogeneous GLDE \eqref{LhGODE} is 
	\begin{itemize}
		\item Variationally stable if and only if it is stable with respect to perturbations.
		\item Variationally attracting if and only if it is attractive with respect to perturbations.
		\item Variationally asymptotically stable if and only if it is asymptotically stable with respect to perturbations. 
	\end{itemize}
\end{theorem}

The following theorem is also due to Schwabik and characterise the variational stability with respect to the transition matrix $U$ given by \eqref{FO}. We will include the original proof for the reader's convenience, see \cite[p.405]{SCHWABIK2}.

\begin{theorem}\label{VS}
 Let $A\in BV_{loc}([0,+\infty),\mathcal{L}(\rn))$ be continuous from the left and condition {\rm (D)} holds. Then, the trivial solution of the homogeneous GLDE \eqref{LhGODE} is variationally stable if and only if the transition matrix $U(t,s_0)$ is bounded for all $t\geq s_0\geq0$.
\end{theorem}
\begin{proof}
Assume that there exists a constant $C>0$ such that $\|U(t,s_0)\|\leq C$, for all $t\geq s_0\geq 0$. Since Theorem~\ref{evsp}, we will prove that the trivial solution of \eqref{LhGODE} is stable with respect to perturbations. Consider a function $y(\cdot,s_0,y_0)$ which is a solution of the perturbed linear equation \eqref{linearP} with initial condition $y(s_0,s_0,y_0)=y_0$. 

From equation \eqref{ibpf2}, the left continuity of $A$, and Theorem~\ref{pU} item (e)-(ii), it follows that for every $t\geq s_0$ 
\begin{equation*}
\begin{split}
\|y(t,s_0,y_0)\|&=\left\|U(t,s_0)y_0 + \int_{s_0}^{t}U(t,s){\rm d}[P(s)-P(s_0)] + \sum_{s_0\leq \tau < t}\Delta^{+}U(t,\tau)\Delta^{+}P(\tau)\right\|\\
&\leq C\|y_0\| + C\text{var}_{s_0}^{t}(P)+2C\text{var}_{s_0}^{t}(P)\\
&\leq C\|y_0\|+3C\sup_{t\geq s_0}\text{var}_{s_0}^{t}(P).
\end{split}
\end{equation*}
Therefore, for a given $\varepsilon>0$, we can choose $\delta=\dfrac{\varepsilon}{4C+1}>0$, which satisfies
\[
\|y_0\|<\delta \; \text{ and } \sup_{t\geq s_0}\text{var}_{s_0}^{t}(P)<\delta \; \Longrightarrow \|y(t,s_0,y_0)\|<\varepsilon, \qquad \text{for all } t\geq s_0.
\]
Hence, the trivial solution of \eqref{LhGODE} is variationally stable. 

On the other hand, assume that \eqref{LhGODE} is variationally stable. For $\varepsilon=1$, there exists a $\delta>0$ such that if $x\colon[s_0,+\infty)\to\rn$ is a solution of \eqref{LhGODE}, i.e. $\text{var}_{s_0}^{t}(x(s)-\int_{s_0}^{s}{\rm d}[A(\xi)]x(\xi))=0$ for all $t\geq s_0$, with $\|x(s_0)\|<\delta$, then $\|x(t,s_0,x(s_0))\|<1$ for all $t\geq s_0$. 

Therefore, for an arbitrary $z\in\rn$ with $\|z\|\leq1$, we have $\|U(t,s_0)(\delta /2) z\|<1$ for all $t\geq s_0\geq0$, because $U(t,s_0)(\delta/2)z$ is a solution of \eqref{LhGODE} with initial condition $(\delta/2)z$. Hence, we obtain 
\[
\|U(t,s_0)\|=\sup_{\|z\|\leq1}\|U(t,s_0)z\|\leq 2/\delta,
\]
for all $t\geq s_0\geq0$, which implies the boundedness of the transition matrix.
\end{proof}
\begin{remark}
	If we think about the previous Remarks \ref{r4.1} and \ref{r4.2}, along with a careful reading of the results obtained in Theorem~\ref{US} and Theorem~\ref{UAS}, it seems reasonable to expect a possible characterisation of the variational asymptotic stability in terms of the transition matrix $U$. 
\end{remark}
The following theorem is the main result of this section. We provide a characterisation of the asymptotic variational stability and the global uniform exponential stability, see Definition~\ref{gstability}. 

\begin{theorem}\label{VAS}
	Let $A\in BV_{loc}([0,+\infty),\mathcal{L}(\rn))$ be continuous from the left and condition {\rm (D)} holds. Then, the trivial solution of the homogeneous GLDE \eqref{LhGODE} is variationally asymptotically stable if and only if there exist positive constants $\alpha,K>0$ such that \eqref{exp} holds for all $t\geq s_0\geq0$. 
\end{theorem}

\begin{proof}
Assume that there exist positive constants $\alpha,K>0$ such that \eqref{exp} holds for all $t\geq s_0\geq0$. Clearly, the transition matrix is bounded, with $\|U(t,s_0)\|\leq K$, for all $t\geq s_0\geq0$, and thus the trivial solution of \eqref{LhGODE} is variationally stable. 

Since Theorem~\ref{evsp}, we will prove that the trivial solution of  \eqref{LhGODE} is attractive with respect to perturbations. Let $\varepsilon>0$ be fixed, assume that $K>\varepsilon$, and let $y_0\in\rn$. Consider a function $y(\cdot,s_0,y_0)$, which is a solution of the perturbed linear equation \eqref{linearP}  with initial condition $y(s_0,s_0,y_0)=y_0$, where $P\in BV_{loc}([s_0,+\infty),\rn)$ is a left-continuous function on $(s_0,+\infty)$. 

For every $t\geq s_0$ we have the estimate
\begin{equation*}
\begin{split}
\|y(t,s_0,y_0)\|&=\left\|U(t,s_0)y_0 + \int_{s_0}^{t}U(t,s){\rm d}[P(s)-P(s_0)] + \sum_{s_0\leq \tau < t}\Delta^{+}U(t,\tau)\Delta^{+}P(\tau)\right\|\\
&\leq Ke^{-\alpha(t-s_0)}\|y_0\| + 3K\text{var}_{s_0}^{t}(P)\\.
\end{split}
\end{equation*}
Let us consider $\delta_{0}=1$, $0<\gamma<\varepsilon/3K$, and $T(\varepsilon)=[\ln(K/(\varepsilon-3K\gamma))]/\alpha$. Therefore, it follows that for $\|y_0\|<\delta_0$, and $\displaystyle\sup_{t\geq s_0}\text{var}_{s_0}^{t}(P)<\gamma$ we get
\begin{equation*}
\|y(t,s_{0},y_0)\|\leq Ke^{-\alpha(t-s_{0})}+3K\gamma<\varepsilon, \qquad \text{ for all } t\geq s_0+T(\varepsilon).
\end{equation*} 
Hence, we conclude that the trivial solution of \eqref{LhGODE} is attractive with respect to perturbations.

On the other hand, assume that the trivial solution of \eqref{LhGODE} is variationally asymptotically stable. Let $0<\varepsilon<1$ be fixed. From the variational attracting assumption, there exists a $\delta_0>0$ independent of $\varepsilon$, and there is a $T=T(\varepsilon)>0$ such that if $x:[s_0,+\infty)\to\rn$ is a solution of \eqref{LhGODE}, i.e. $\text{var}_{s_0}^{t}(x(s)-\int_{s_0}^{s}{\rm d}[A(\xi)]x(\xi))=0$ for all $t\geq s_0$, with $\|x(s_0)\|<\delta_0$ then 
\[
\|x(t,s_0,x(s_0))\|=\|U(t,s_0)x(s_0)\|<\dfrac{\varepsilon\delta_0}{2}, \qquad \text{ for all } t\geq s_0+T(\varepsilon). 
\]
Hence, for an arbitrary $z\in\rn$ with $\|z\|\leq1$ we obtain $\|U(t,s_0)(\delta_{0}/2)z\|<\varepsilon\delta_{0}/2$ for all $t\geq s_{0} + T$, which in turn implies that $\|U(t,s_0)z\|<\varepsilon$ for all $t\geq s_{0} + T$. And we get 
\begin{equation*}
\|U(t,s_0)\|=\sup_{\|z\|\leq1}\|U(t,s_0)z\|<\varepsilon, \qquad \text{ for all } t\geq s_{0} + T. 
\end{equation*}
 Now, from the variational stability assumption of \eqref{LhGODE} and Theorem~\ref{VS}, we deduce the existence of a positive constant $C>0$ such that $\|U(t,s_0)\|\leq C$ for all $t\geq s_0\geq0$, and the rest of the proof follows exactly analogue to the final part of the proof given in Theorem~\ref{UAS}. 
\end{proof}

\begin{remark}
	The notion of uniform stability of the trivial solution of the homogenous GLDE \eqref{LhGODE} given in Section 3 is equivalent to the notion of variational stability, both concepts are comparable to the uniform boundedness of the transition matrix $U$ defined by \eqref{FO}, namely, a bound independent of the initial time $s_0\geq0$. In addition, the uniform asymptotic stability notion from Section~3  is equivalent to the variational asymptotic stability of the trivial solution of the homogeneous GLDE \eqref{LhGODE}, and as we seen in Theorem~\ref{UAS} and Theorem~\ref{VAS}, both of these concepts are equivalent to the uniform exponential decay of the transition matrix described by \eqref{exp}, and consequently coincide with the global uniform exponential stability notion given in Definition~\ref{gstability}. 
\end{remark}

\end{document}